\documentclass[12pt]{amsart}
\usepackage{calrsfs}
\usepackage{soul}
\usepackage[greek,english]{babel}
\usepackage[iso-8859-7]{inputenc}
\usepackage{graphicx}
\usepackage{hyperref}
\usepackage{amsthm}
\usepackage{amssymb}
\usepackage{multirow}
\usepackage{tikz-cd}
\usepackage{amsmath}
\usepackage{todonotes}
\usepackage{amsbsy}
\usepackage[all]{xy}
\usepackage{changes}
\usepackage{enumerate}
\usepackage{comment}
\usepackage{mathtools}
\usepackage{changes}
\usepackage{mdframed}
\usepackage{faktor} 
\usepackage{stackrel}
\usepackage[shortlabels]{enumitem}
\usepackage{ bbold }
\usepackage{hhline}

\usepackage{stmaryrd}

\definechangesauthor[color=orange,name={Peter Symonds}]{PS}
\definechangesauthor[color=red, name={Kostas Karagiannis}]{KK}

\usepackage[margin=1 in]{geometry}

\newtheorem{theorem}{Theorem}[subsection]
\newtheorem*{theorem*}{Theorem}
\newtheorem{lemma}[theorem]{Lemma}
\newtheorem{corollary}[theorem]{Corollary}
\newtheorem{proposition}[theorem]{Proposition}

\theoremstyle{definition}
\newtheorem{example}[theorem]{Example}
\newtheorem{remark}[theorem]{Remark}
\newtheorem{definition}[theorem]{Definition}

\newcommand{\Spec}{{\rm Spec }}

\newcommand{\eva}{\rm ev}

\newcommand{\suchthat}{\;\ifnum\currentgrouptype=16 \middle\fi|\;}

\newcommand{\aprod}{\mathop{\operator@font \hbox{\Large$\ast$}}}

\DeclareMathOperator{\ind}{\mathbf{I}}
\DeclareMathOperator{\coind}{\mathbf{C}}
\DeclareMathOperator{\res}{\mathbf{R}}

\DeclareMathOperator{\Hom}{Hom}

\DeclareMathOperator{\Tr}{Tr}
\DeclareMathOperator{\soc}{soc}
\DeclareMathOperator{\ide}{Id}

\DeclareMathOperator{\infl}{Inf}

\date{\today}

\keywords{Group schemes, transfer, norm}
\subjclass{14L15}

\title{Transfer and Norm for Finite Group Schemes}

\author[K. Karagiannis]{Kostas Karagiannis }
\address{
Department of Mathematics,			
University of Manchester,
Manchester M13 9PL, 
United Kingdom.}
\email{konstantinos.v.karagiannis@gmail.com}

\author[P. Symonds]{Peter Symonds }
\address{
Department of Mathematics,			
University of Manchester,
Manchester M13 9PL, 
United Kingdom.}
\email{peter.symonds@manchester.ac.uk}

\date \today



\begin{document}

\begin{abstract}
	We develop the theory of transfer and norm maps for finite group schemes, extending classical results from finite group theory to a context where induction and restriction are not necessarily bi-adjoint. In the additive setting, we construct a  transfer map for both modules and $\rm Ext $ groups and prove that its surjectivity characterizes relative projectivity, establishing a  generalization of Higman's criterion. In the multiplicative setting, we define a relative norm map for algebras with a group scheme action. We compare this norm with other versions in the literature, proving that it coincides with Mumford's norm for finite morphisms and that, on fields, it is a power of the classical field norm.
\end{abstract}
\maketitle

\section{Introduction}

\subsection{Overview} The transfer map in group representation theory grew out of Schur's transfer map in group theory and the Reynolds operator in invariant theory. It is an additive operator on modules or ${\rm Ext}$ groups that  has developed into an indispensable and ubiquitous tool.  There are also various multiplicative versions in different contexts, such as the field norm, Mumford's norm and the Evens norm in group cohomology. The objective of this paper is to develop systematically a theory of trace and norm maps in the context of finite group schemes over a field, in the hope and expectation that it will prove as useful as the original.

 For group schemes, a definition involving summing or multiplying over the orbit of an element will clearly not work and we need a different approach, based on induced modules. However, a fundamental difference between the representation theory of finite group schemes and that of finite groups is that, in the more general context, the left and right adjoints of restriction to a subgroup  are not always isomorphic, although they are related by the Wirthm\"uller isomorphism; this introduces a twist by a one-dimensional representation. We show how a transfer can still be defined and that it has many of the properties that we desire, for example the relation with relative projectivity. There is also a version of Mackey's double coset formula, but it is necessarily weaker than one might have hoped. We also define a multiplicative norm for a group scheme action on an algebra and relate it to other norms in the literature.

\subsection{Outline} Induction and coinduction are introduced in Section \ref{sec:ind}, while Section \ref{sec:modular} studies modular functions, a measure of the extent to which they fail to be isomorphic. The precise relationship between induction and coinduction is given in Section \ref{sec:wirth}, via the so-called Wirthm\"uller isomorphism. Section \ref{sec:deftransfer} proceeds with the definition of a transfer map that generalizes the one for finite groups and in Section \ref{sec:proptransfer} its basic properties are developed. Section \ref{sec:Higman} investigates the connection between transfer and relative projectivity and a variant of Higman's Criterion is proved, while a generalization of Mackey's double coset formula is presented in Section \ref{sec:Mackey} . The definition of the transfer and its basic properties are extended to cohomology and Ext groups in Section \ref{sec:coh}. In Section \ref{sec:notation} our attention switches to an action of a finite group scheme on a commutative algebra. Following a study of invariants of infinitesimal group schemes in Section \ref{sec:norminf}, a version of the relative norm map is defined in Section \ref{sec:defnorm}, with its basic properties developed in Section \ref{sec:propnorm}. Section \ref{sec:fieldnorm} treats the case in which the object acted upon is a field and investigates connections with the classical field norm, while Section \ref{sec:Mumford} compares the relative norm defined in this paper to a construction of Mumford.

\subsection*{Acknowledgments}
This research was supported by the Engineering and Physical Sciences Research Council (Grant no. EP/V036017/1).

\section{Preliminaries }\label{sec:prel}

\subsection{Notation \& Conventions}\label{sec:not}
Let $G$ be a finite group scheme over a field $k$ of characteristic $p>0$. The coordinate ring of $G$ is a finite dimensional, commutative, Hopf $k$-algebra, denoted by $k[G]$. Its $k$-dual, $kG=k[G]^*$, is also endowed with the structure of a Hopf $k$-algebra and will be referred to as the group algebra of $G$ (\cite[9.1.3]{MR1243637} , where it is denoted $M(G)$). Left $G$-modules in the functorial sense are identified with left $kG$-modules, see \cite[Theorem 2.1]{MR1885648}, similarly on the right. Unless otherwise specified, modules are assumed to be left modules, but we can switch between left and right modules using the antipode and this will be assumed to have been done when it is necessary for a formula to make sense. Unsubscripted Hom sets and tensor products are meant to be taken over the ground field $k$.

\subsection{Induction \& Coinduction}\label{sec:ind}
Given a closed subgroup scheme $H$ of $G$, let $\mathbf{R}_H^G(-)$ be the restriction functor  from the category of $G$-modules to the category of $H$-modules; when it is clear from the context, the restriction of a $G$-module $M$ to $H$ will also be denoted by $M$. We regard $kG$ as a $G{-}G$-bimodule; when we wish to keep track of which copy of $G$ acts on which side we will write something such as ${}_{G_1}kG_{G_2}$, where $G_1$ and $G_2$ are the two copies of $G$. The subspace of $M$ of invariants under the action of $G$ will be denoted by ${}^GM$ or $M^G$, according to the side of the action. 

The algebra $kG$ is a Frobenius algebra, so has many of the properties expected of an ordinary group algebra; however, it might not be symmetric.

In particular, we have a left adjoint to restriction, known as coinduction, $\coind_H^G=k{}_GG_H\otimes_{kH} -$ and a right adjoint, known as induction, $\ind^G_H = \Hom_{kH}(k{}_HG_G,-)$. These functors satisfy the usual projection formulas $M \otimes \coind^G_H N \cong \coind^G_H (\res^G_H M \otimes N)$ and $M \otimes \ind^G_H N \cong \ind^G_H (\res^G_H M \otimes N)$.

This is the conventional notation for group schemes. In group representation theory the names of induction and coinduction are usually switched.

Denote by $(\eta_{G,H},\epsilon_{G,H})$ the unit and counit associated to the adjunction between $\coind_H^G$ and $\res_H^G$ and by $(\eta_{H,G},\epsilon_{H,G})$ that between $\res_H^G$ and $\ind_H^G$; explicitly, if $M$ is a $G$-module and $N$ is an $H$-module, then

\begin{align}\label{eq:unit-counit}
 &\eta_{G,H}(N) \colon N\rightarrow {\bf R}_H^G\coind_H^G(N),
\quad \;\;\; \epsilon_{G,H}(M) \colon \coind_H^G{\bf R}_H^G (M)\rightarrow M&
 \\ \nonumber
 &\eta_{H,G}(M) \colon M \rightarrow \ind_H^G{\bf R}_H^G(M),
\quad \epsilon_{H,G}(N) \colon {\bf R}_H^G \ind_H^G(N)\rightarrow N .&
\end{align}

In our case there are standard choices for these, which we will always use.

The three functors are transitive: if $K$ is a closed subgroup scheme of $H$, then there are standard isomorphisms $\coind_H^G\circ\coind_K^H\cong\coind_K^G,\; \mathbf{R}_K^H\circ \mathbf{R}_H^G \cong\mathbf{R}_K^G$ and $\ind_H^G\circ \ind_K^H\cong\ind_K^H$.

\subsection{Integrals and Modular functions} \label{sec:modular}

By the Larson-Sweedler Theorem for finite dimensional Hopf algebras \cite[3.2]{MR2894798}, \cite[2.1.3]{MR1243637}, the spaces of invariants under either the left or the right action of $G$ on $kG$ both have dimension one (see \cite[I.8.7]{MR2015057} for group schemes). These are known as the spaces of left or right integrals \cite[2.1.1]{MR1243637} (or invariant measures in \cite[I.8.8]{MR2015057}). They are modules for the action of $G$ on the other side, and we define a left module ${}_G\delta = kG^G$ and a right module $\delta_G = {}^GkG$. The antipode takes one vector space to the other and allows us to identify ${}_G\delta = kG^G$ and $\delta_G = {}^GkG$, provided that we remember our convention that swapping sides involves the antipode (even though this is not necessary for 1-dimensional modules).

\begin{definition}\label{def:delta}
In the above notation:

\begin{enumerate}

\item Note that the definition of ${}_G\delta$ contains more information than just the isomorphism class of a module; there is a specific injection of left $G$-modules ${}_G\delta \rightarrow kG$. 

\item The character of ${}_G\delta$ is known as the modular function. Often ${}_G\delta$ is just defined to be this character, for example in \cite[I.8.8]{MR2015057}, but the extra precision will ensure that various constructions are canonical rather than just defined up to scalar multiple.

\item When $G$ is a finite group, we have ${}_G\delta$ = $\delta_G$ as vector spaces, with canonical basis element  $\Sigma_G = \sum_{g\in G}g \in kG$. When $G$ is \'{e}tale, $\Sigma_G$ is defined when we pass to the algebraic closure and it is invariant under the absolute Galois group, so still gives us a canonical basis element for ${}_G\delta$.
 
\item One says that $G$ is {\em unimodular} if ${}_G\delta$ is trivial. For an  \'{e}tale  group scheme there is a canonical isomorphism from ${}_G\delta$ to $k$ that sends $\Sigma_G$ to 1, but in general this isomorphism is not canonical. 

\item For a 1-dimensional module $L$, left or right, we write $L^{-1}$ for the dual of $L$, considered as a module for the same side via the antipode. Thus $L \otimes L^{-1}$ is trivial.

\item The {\em relative dualizing object} of $H$ in $G$ is defined to be the left $H$-module $\omega _{H,G}=\mathbf{R}_H^G({}_G\delta) \otimes {}_H\delta^{-1}$. 
\end{enumerate}
\end{definition}

 When $K$ and $H$ are closed subgroup schemes of $G$ with $K\leq H$ we have a canonical isomorphism $\omega_{K,G} \cong \res^H_K(\omega_{H,G}) \otimes_k \omega_{K,H}$ which will be abbreviated as $
\omega_{K,G}=\omega_{H,G}\omega_{K,H}$.

\section{The transfer}\label{sec:transfer}

As in the previous section, $G$ is a finite group scheme over a field $k$ of characteristic $p>0$, $H$ is an arbitrary closed subgroup scheme, and $\coind_H^G \dashv \mathbf{R}_H^G\dashv \ind_H^G$ is the coinduction-restriction-induction adjoint triple.

\subsection{Definition of the transfer}\label{sec:deftransfer}
In the notation of the previous section, we have 

\[ {}_G \delta = kG^G \subseteq kG^H \cong kG \otimes_H kH^H = kG \otimes_H {}_H \delta . \]

This gives us a canonical injection $\tilde{t}_{H,G} \! : {}_G \delta \rightarrow \coind^G_H {}_H \delta $. Sometimes it is convenient to tensor this with ${}_G \delta^{-1}$ to obtain $t_{H,G} \! : k \rightarrow {}_G \delta^{-1} \otimes  \coind^G_H {}_H \delta \cong \coind^G_H ({}_G \delta^{-1} {}_H \delta) \cong \coind ^G_H \omega^{-1}_{H,G}$.

\begin{definition}\label{def:transfer}
	Let $L, M$ and $N$ be left $G$-modules. The transfer from $H$ to $G$ on $L$ is the map of $k$-modules
	
	\[\Tr _H^G  \! : \Hom_H( \omega^{-1}_{H,G}, L) \rightarrow \Hom_G(k,L)\]
	
	given by
	
	\[ \Hom_H (\omega_{H,G}^{-1}, L) \cong \Hom_G(\coind^G_H \omega_{H,G}^{-1},L) \xrightarrow{t^*_{H,G}} \Hom_G(k,L).\]
	
	In formulas this becomes $\Tr^G_H(f) = \epsilon_{G,H}(L) \circ (\ide_{kG} \otimes f) \circ t^G_H$ for $f
	\in \Hom_H (\omega_{H,G}^{-1}, L)$, where $\epsilon_{G,H}(L) \! : \coind ^G_H  \res ^G_H L \rightarrow L$, $g \otimes x \mapsto gx$ is the counit of the adjunction.
	
	Since $\Hom_H( \omega^{-1}_{H,G}, L) \cong {}^H(\omega_{H,G} \otimes L ) $ and $\Hom_G(k,L) \cong\;^G L$, we can also regard the transfer as a map 
	
		\[ \Tr_H^G \! :  {}^H(\omega_{H,G} \otimes L ) \rightarrow {}^G L. \]
	
	We can also use $\tilde{t}^G_H$ to obtain a map
	
		\[ \tilde{\Tr}_H^G \! :  {}^H({}_H \delta \otimes M ) \rightarrow {}^G({}_G \delta \otimes M). \]
		
		Replacing $L$ by $\Hom (M,N)$, we obtain
		
			\[ \Tr_H^G \! :  {}^H(\omega_{H,G} \otimes \Hom (M,N) ) \rightarrow {}^G \Hom (M,N). \]
			
			The domain can also be written as $\Hom_H(M,\omega _{H,G} \otimes N)$ or $\Hom_H(\omega _{H,G}^{-1} M \otimes N)$ and the codomain is, of course, $\Hom_G(M,N)$.

\end{definition}
We now pause to make some comments on the definition.

\begin{enumerate}

\item If $G$ is a finite group and we use the basis elements $\Sigma_G$ and $\Sigma_H$ for ${}_G \delta$ and ${}_H\delta$ as before, then ${}_G \delta$ and ${}_H\delta$ are canonically isomorphic to $k$ and we obtain 	$\Tr_H^G \! :  {}^HL  \rightarrow {}^G L$. Furthermore, ${}_G \delta = \sum_{g \in G/H} g \otimes {}_H \delta$, so we obtain the usual transfer as in \cite[3.6.2]{benson}. \\

\item If $G$ and $H$ are unimodular (or just if $\omega _{H,G}$ is trivial) we have ${}^H(\omega_{H,G} \otimes M) \cong {}^HM$, but not canonically. We obtain a map $\Tr_H^G \! :  {}^H L  \rightarrow {}^G L$, but it is only determined up to non-zero scalar multiple.

\item Variants of the transfer have been studied in the literature by several authors. \cite[2.11]{MR1375579} treated the case of unimodular Hopf algebras, whereas that of symmetric Frobenius algebras can be found in \cite[p.57]{MR1102964} and \cite[2.4]{MR2863463}. \cite[6.6]{MR2526081} gave a more general treatment in the context of triangulated categories. In all the above cases, induction and restriction are assumed to be bi-adjoint. To the best of our knowledge, the only two sources where this assumption is dropped are \cite[\S 6]{MR190183} and \cite[\S 2]{MR585205}; the definition above is motivated by these two papers.

\end{enumerate}

\subsection{Properties of the transfer}\label{sec:proptransfer}

The results below can be seen as generalizations of well-known properties of the transfer in the context of finite groups, cf.\ \cite[3.6.3]{benson}, to that of finite group schemes.

\begin{proposition}\label{pr:properties-tr}
Let $L,L',M',M,N',N$ be $G$-modules, and let $K,H$ be closed subgroup schemes of $G$ with $K\leq H$.

\begin{enumerate}
\item 
\begin{enumerate}
	\item \label{tr1a}
	If $d \in \Hom_G(L,L')$ and $ s \in \Hom_H(\omega^{-1}_{H,G}, L)$ then $\Tr^G_H (d \circ s) = d \circ \Tr^G_H(s)$.
	\item \label{tr1b}
If $f \in \Hom_H (M,  \omega_{H,G} \otimes N), \; g \in \Hom_G (M',M)$ and $h \in \Hom_G (N,N')$, then 
\[
h \circ \Tr_H^G(f)\circ g = \Tr_H^G( ( \ide_{\omega_{H,G}} \otimes h) \circ f \circ g).
\]
\end{enumerate}

\item \label{tr2} $\tilde{\Tr}^G_K = \tilde{\Tr}^G_H \circ \tilde{\Tr}^H_K \! :  {}^K({}_K \delta  \otimes L) \rightarrow {}^G({}_G \delta  \otimes L)$ and 
$\Tr^G_K = \Tr^G_H \circ (\ide _{\omega_{H,G}} \otimes \Tr^H_K) \! :  {}^K( \omega_{K,G}  \otimes L) \rightarrow {}^GL$. 

\item \label{tr3} If $\ell /k$  is a field extension, and $G_\ell=G\times_{\Spec (k)} \Spec (\ell),\;H_\ell=H\times_{\Spec (k)} \Spec (\ell)$, then $  \Tr_{H_\ell}^{G_\ell} = \Tr_{H}^G \otimes \ell \! : 
{}^H(\omega_{H_\ell,G_\ell} \otimes L) \rightarrow {}^G(\ell \otimes L)$.

\end{enumerate}
\end{proposition}

\begin{proof}
	
	Part \ref{tr1a} follows from the definitions, since the definition of $\Tr$ only uses the first argument of $\Hom$. Part \ref{tr1b} then follows by setting $L= \Hom(M,N)$ and $L'=\Hom(M',\omega_{H,G}N')$.
	
	Part \ref{tr2} is the observation that the composite
	\[ {}_G \delta \xhookrightarrow{\tilde{t}_{H,G}} kG \otimes _H {}_H \delta \xhookrightarrow{\ide_{kG} \otimes \tilde{t}_{H,G}} kG \otimes_{kH}kH \otimes _{kK} {}_K \delta \cong kG \otimes_K {}_K \delta \]
	is just $\tilde{t}_{K,G}$.
	Finally,  part \ref{tr3} is just the observation that everything we have done commutes with base change.
\end{proof}

\begin{corollary} 
If $R$ is a $k$-algebra and $M$ and $M'$ are $R$-modules in a way that commutes with the action of $G$, then the transfer is a map of $R$-modules.
\end{corollary}

\subsection{The Wirthm\"uller isomorphism}\label{sec:wirth}

\begin{theorem}\label{th:wirth} 
	There exists a natural isomorphism of functors
	\[
	W_{H,G} \colon  \ind_H^G(-)\overset{\cong}{\longrightarrow}\coind_H^G(-\otimes\omega_{H,G}^{-1})
	\]
	called the {\em Wirthm\"uller isomorphism}.
\end{theorem}

This can be found in \cite[I.8.17]{MR2015057}.  We discuss it further in Section~\ref{wirt}.

\begin{remark}
	We could have used the Wirthm\"uller isomorphism to produce a non-zero map as the composite $k \xrightarrow{\eta_{H,G}} \ind ^G_H (k) \xrightarrow{W_{H,G}} \coind ^G_H (\omega^{-1}_{H,G})$. Since $\dim \Hom_G(k,\coind ^G_H \omega^{-1}_{H,G}) = \dim \Hom_G(k, \ind ^G_H k) = \dim \Hom_H(k,k)=1$, this is equal to $t_{H,G}$ up to a non-zero scalar multiple. 
\end{remark}

\begin{lemma}\label{l:P1}
	The $H$-homomorphism $\res ^G_H (t^G_H)$ is a split injection.
\end{lemma}

\begin{proof}
One of the unit-counit identities for the adjunction between $\res^G_H$ and $\ind^G_H$ when applied to $k$ gives $\ide_{\res^G_H}=\epsilon_{H,G}(\res^G_H(k)) \circ \res^G_H(\eta_{H,G}(k))$. This shows that $\epsilon_{H,G}(\res^G_H)$ is split and thus $\res ^G_H (t^G_H)$ is split, by the remark above.
\end{proof}

We can ask whether other choices of 1-dimensional representations can also yield a map like $t_{H,G}$. We say that ${}_G \hat{\delta}$ and ${}_H \hat{\delta}$ are compatible if there exists a non-zero map ${}_G \hat{\delta} \hookrightarrow \coind^G_H {}_H \hat{\delta}$. Tensoring with ${}_G \hat{\delta}^{-1}$ gives a non-zero map $ k \hookrightarrow \coind^G_H {}_G \hat{\delta}^{-1} {}_H \hat{\delta}$. Set $\hat{\omega}_{H,G} = {}_G \hat{\delta}^{-1}{}_H \hat{\delta}$. We now have a non-zero map $k \hookrightarrow \coind^G_H \hat{\omega}_{H,G} \cong  \ind^G_H \omega^{-1}_{H,G} \hat{\omega}_{H,G}$ and, by adjunction, a non-zero map of $H$-modules $k \hookrightarrow   \omega^{-1}_{H,G} \hat{\omega}_{H,G}$. It follows that $\hat{\omega}_{H,G} \cong \omega_{H,G}$, so $\omega_{H,G}$ does not vary, up to isomorphism. A variation on this argument also shows that the $\omega_{H,G}$ from the Wirthm\"uller isomorphism is isomorphic to the one we originally constructed.

Any pair  ${}_G \hat{\delta}$, ${}_H \hat{\delta}$ can be used, as long as ${}_G \hat{\delta}^{-1}{}_H \hat{\delta} \cong \omega_{H,G}$. We see that any ${}_G \hat{\delta}$ is possible by setting ${}_H \hat{\delta}=   \omega^{-1}_{H,G}  {}_G \hat{\delta} $.  However, not all ${}_H \hat{\delta}$ might be possible, not even ${}_H \hat{\delta} \cong k$ (see  \cite[II 3.4]{MR2015057}).

\begin{lemma}
	\label{inf}
	
	Suppose that we have closed subgroup schemes $N \leq H \leq G$, with $N$ normal in $G$, and we set $\bar{G} = G/N$, $\bar{H}=H/N$. Then $\omega_{H,G} = \infl^H_{\bar{H}} \omega_{\bar{H}, \bar{G}}$ and $t_{H,G} = \infl^G_{\bar{G}} t_{\bar{G},\bar{H}}$ up to a non-zero scalar. 

\end{lemma}

Here $\infl$ denotes the usual inflation map.

\begin{proof}
The first part is because $t_{\bar{H},\bar{G}}$ inflates to a non-zero map $k \hookrightarrow  \coind^G_H  \infl^H_{\bar{H}} \omega_{\bar{H}, \bar{G}}^{-1}$ hence $k$ and $\infl^H_{\bar{H}} \omega_{\bar{H}, \bar{G}}^{-1}$ qualify as ${}_G \hat{\delta}$ and ${}_H \hat{\delta}$ in the discussion above. This also proves the second part. 
\end{proof}

The $\delta$ might not correspond under inflation.

\begin{corollary}\label{cor:extra} In the circumstances of the lemma and in the case that $G^o=H^o$, we can take $N=G^o$. Then $\bar{G}$ and $\bar{H}$ are \'{e}tale, so $\omega_{\bar{H}, \bar{G}}$ is trivial with a canonical generator; this inflates to $\omega_{H, G}$.
	\end{corollary}

\begin{corollary} In the circumstances of the lemma, $\res ^H_N \omega_{H,G}$ is trivial and any compatible pair ${}_G \hat{\delta}$ and ${}_H \hat{\delta}$ are isomorphic on restriction to $N$. 
	\end{corollary}
	
	\begin{corollary}
		In the circumstances of the lemma, $\Tr^G_H \circ \infl^H_{\bar{H}} = \infl^G_{\bar{G}} \circ \Tr^{\bar{G}}_{\bar{H}}$ up to non-zero scalar multiple.
	\end{corollary}

\subsection{Relative projectivity}\label{sec:Higman}
Throughout this section, $\mathcal{C}$ is a collection of closed subgroup schemes of $G$; it might be infinite.

\begin{definition}\label{def:relproj}
A $G$-module $M$ is said to be {\em projective relative to} $\mathcal{C}$ if, for any maps of $G$-modules $q \! : U \rightarrow V$ and $r \! : M \rightarrow V$ such that, for any $H \in \mathcal{C}$, the map $r$ lifts to a map of $H$-modules $r_H \! : M \rightarrow U$ such that $r=qr_H$, it follows that $q$ lifts to a map of $G$-modules $s \! : M \rightarrow U$ such that $r=qs$.  
\end{definition}

 Equivalent characterizations of relative projectivity first appeared in work of Higman \cite{MR67895}, \cite{MR87671} and Ikeda \cite{MR55326} and were generalized by Morita \cite{MR190183}, Chouinard \cite{MR585205} and Brou\'e \cite{MR2526081}. The proof of the following version of Higman's Criterion is essentially a synthesis of arguments from the above sources, whose treatment applies to collections consisting of a single subgroup scheme, and
 \cite[3.5.8]{MR2266036}, which allows for arbitrary collections but only in the context of constant group schemes.
 
 \begin{lemma}\label{l:P2} If $H \leq G$ and the $G$-module $M$ is a summand of $\ind^G_H(N)$ or $\coind^G_H(N)$ for some $H$-module $N$, then $\Tr _H^G  \! : \Hom_H( \omega^{-1}_{H,G}, M) \rightarrow \Hom_G(k,M)$ is surjective.
 	\end{lemma}
 	
 	\begin{proof}
 		By the Wirthm\"{u}ller isomorphism, we only need to check the case with $\ind^G_H(N)$. By naturality of the transfer, it suffices to consider the case $M=\ind^G_H(N)$.
 		
 		We want $\Hom_G(\coind^G_H( \omega ^{-1}),\ind^G_H(N)) \xrightarrow{(t^G_H)^*} \Hom_G(k, \ind^G_H(N))$ to be onto. By a version of the tensor-Hom identity, this is $I^G_H$ applied to $\Hom_H(\res^G_H \coind^G_H( \omega ^{-1}),N) \xrightarrow{(t^G_H)^*} \Hom_H(k, N)$, which is split by Lemma~\ref{l:P1}.
 	\end{proof}

\begin{proposition}[Higman's Criterion]\label{prop:Higman}
Let $\mathcal{C}$ be a collection of closed subgroup schemes of $G$. The following conditions on a $G$-module $M$ are equivalent.
\begin{enumerate}
\item $M$ is isomorphic to a direct summand of $\bigoplus_{H\in \mathcal{C}}\coind_H^G (M)$.
\item $M$ is isomorphic to a direct summand of $\bigoplus_{H\in \mathcal{C}}\coind_H^G(N_H)$, for some $H$-modules $N_H$.
\item There exists a collection $\{f_H\in {\rm Hom}_H(M\otimes\omega_{H,G}^{-1},M)\}_{H\in\mathcal{C}}$ such that
\begin{enumerate}
\item  For any $m\in M$, ${\rm Tr}_{H,\lambda_H}^G\left(f_H\right)(m)= 0$ for all but finitely many $H\in\mathcal{C}$.
\item $\sum_{H\in\mathcal{C}}{\rm Tr}_{H}^G\left(f_H\right)(m)=m$ for all $m\in M$.
\end{enumerate}  
\item $M$ is projective relative to $\mathcal{C}$.
\end{enumerate}
\end{proposition}

\begin{proof}
The implication $(1)\Rightarrow (2)$ is trivial and $(4)\Rightarrow (1)$ is straightforward.\\
$(2)\Rightarrow (3)$: Let $E=\bigoplus_{H\in \mathcal{C}} \coind_H^G(N_H)$, where each $N_H $ is an $H$-module. 
From Lemma~\ref{l:P2}, for each $H$ we obtain  $g_H\in {\rm Hom}_H(\coind_H^G(N_H )\otimes\omega_{H,G}^{-1}, \coind_H^G(N_H))$ with transfer $\ide_{N_H}$.
 Thus
\[
{\rm Id}_E
=\bigoplus_{H\in\mathcal{C}}{\rm Id}_{ \coind_H^G\left(N_H\right)}
=\bigoplus_{H\in\mathcal{C}}{\rm Tr}_{H}^G\left(g_H\right).
\]

Let $M$ be a direct summand of $E$, and let $f_H\in{\rm Hom}_H(M\otimes \omega_{H,G}^{-1},M)$ be the composition
\[
M\otimes\omega_{H,G}^{-1}\overset{\iota\otimes{\rm id}}{\hookrightarrow}
E\otimes\omega_{H,G}^{-1}\overset{\pi_H\otimes {\rm id}}{\twoheadrightarrow}
\coind_H^G\left(N_H\right)\otimes\omega_{H,G}^{-1}\overset{g_H}{\rightarrow}
\coind_H^G\left(N_H\right)\overset{\iota_H}{\hookrightarrow}
E\overset{\pi}{\twoheadrightarrow}
M.
\]
where $\iota,\iota_H$ are inclusions and $\pi,\pi_H$ are projections.  By Proposition \ref{pr:properties-tr}, ${\rm Tr}_{H}^G\left(f_H\right)$ is the composition
 \[
M \overset{\iota}{\hookrightarrow}
E \overset{\pi_H}{\twoheadrightarrow}
\coind_H^G \left(N_H\right) \overset{{\rm Tr}_{H}^G\left(g_H\right)}{\longrightarrow}
\coind_H^G \left(N_H\right)\overset{\iota_H}{\hookrightarrow}
E\overset{\pi}{\twoheadrightarrow}
M,
\]
and so $\left\{f_H\in{\rm Hom}_H\left(M\otimes\omega_H^{-1},M\right)\right\}_{H\in\mathcal{C}}$ satisfies condition {\em (3a)}. For condition {\em (3b)}, observe that
\[
\sum_{H\in\mathcal{C}}\left(\iota_H\circ{\rm Tr}_{H}^G\left(g_H\right)\circ\pi_H\right)=
\bigoplus_{H\in\mathcal{C}}\left(\iota_H\circ{\rm Tr}_{H}^G\left(g_H\right)\circ\pi_H\right)=
{\rm Id}_{E}.
\]
Thus, for all $m\in M$,
\[
\sum_{H\in\mathcal{C}}{\rm Tr}_{H}^G\left(f_H\right)(m)
=\pi\circ\sum_{H\in\mathcal{C}}\left(\iota_H\circ{\rm Tr}_{H}^G\left(g_H\right)\circ\pi_H\right)\circ \iota (m)
=m.
\]

$(3)\Rightarrow (4)$: With the notation of Definition~\ref{def:relproj}, let $s=\sum_{H\in\mathcal{C}}{\rm Tr}_{H}^G\left(\sigma_H\circ f_H\right)$. Then
\[
q\circ\sum_{H\in\mathcal{C}}{\rm Tr}_{H}^G\left(r_H\circ f_H\right)=
\sum_{H\in\mathcal{C}}{\rm Tr}_{H}^G \left( q \circ r_H \circ f_H \right)=
\sum_{H\in\mathcal{C}}{\rm Tr}_{H}^G\left(f_H\right)={\rm Id}_{M},
\]
as required.
\end{proof}

\subsection{The composition $\Tr ^G_H \res^G_H$}\label{subsec:trres}
For finite groups (i.e.\ constant finite group schemes) there is the well-known formula $\Tr^G_H \res^G_H = |G{:}H| \ide$ on any $G$-module, which has many applications. In this sub-section we investigate the extent to which this holds for group schemes.

Because $\Tr^G_H \res^G_H$ is induced by a homomorphism 
\begin{equation}\label{e:Q1} k \xrightarrow{t_{H,G}} \coind^G_H (\omega_{H,G}^{-1}) \rightarrow \omega_{H,G}^{-1} ,\end{equation}
it is certainly zero if $\omega_{H,G}$ is not trivial. For the rest of this section (apart from Lemma~\ref{lem:coprime}) we assume that $\omega_{H,G}$ is trivial, we choose an isomorphism $\omega_{H,G} \cong k$ and then the composition above becomes multiplication by a scalar $\lambda_{H,G}$. It follows that $\Tr^G_H \res^G_H= \lambda_{H,G} \ide$ on any module. Of course, $\lambda_{G,H}$ depends on the choice of isomorphism $\omega \cong k$, but we are only interested in whether it is zero or not. We write $\lambda \sim \mu $ to indicate that $\lambda$ and $\mu$ are both either zero or both non-zero.

\begin{lemma}\label{l:Q1} Suppose $K \leq H \leq G$.
	\begin{enumerate} 
		\item 
		$\lambda _{K,G} \sim \lambda_{K,H} \lambda_{H,G}$.
	\item
	When $N \unlhd G$ we have $\lambda_{H/N,G/N} \sim \lambda_{H,G}$. 
	\end{enumerate}
\end{lemma}

\begin{proof} This follows from the proof of Lemma~\ref{inf}.
	\end{proof}
	
	\begin{lemma}\label{l:Q2}
		For $H \leq G$, the following conditions are equivalent:
		\begin{enumerate}
		 \item $\lambda_{H,G} \ne 0$,
			\item $t^G_H$ is split injective as a homomorphism of $G$-modules,
			\item $k$ is a summand of $\coind^G_H (k)$,
			\item every $G$-module is projective relative to $H$.
		\end{enumerate}
	\end{lemma}
	
	\begin{proof}
	$(1) \Rightarrow (2) \Rightarrow (3)$ from the definitions.
	
	$(3) \Rightarrow (1)$: From properties of adjoints, we have $$\dim \Hom_G(k,\coind^G_H(k)) = \dim \Hom_G(\coind^G_H(k),k)=1.$$ The morphisms in eq.\ (\ref{e:Q1}) are non-zero and thus agree with the ones exhibiting $k$ as a summand of $\coind^G_Hk$ up to scalar multiple.
	
	$(3) \Rightarrow (4)$: If $k$ is a summand of $\coind^G_H(k)$, then $M$ is a summand of $M \otimes \coind^G_H (k)$, which is isomorphic to $\coind^G_H \res^G_H (M)$.
	
	$(4) \Rightarrow (3)$ by \ref{prop:Higman}.
	\end{proof}
	
	\begin{lemma} \label{lem:coprime}
	If $H \leq G$ and $p$ does not divide $|G{:}H|$, then the assumption that $\omega_{H,G} \cong k$ is automatically satisfied and $\lambda_{H,G} \ne 0$.
	\end{lemma}
		
		\begin{proof}
		The assumption is satisfied by Corollary~\ref{cor:extra}. The conditions imply that $G^o \leq H$. Thus $\lambda_{H,G} \sim \lambda_{H/G^o,G/G^o}$ by Lemma~\ref{l:Q1}. This extension is \'etale, so the constant group case applies after field extension, hence also before.
	\end{proof}
	
	\begin{proposition}
		For $H \leq G$ the following conditions are equivalent:
		\begin{enumerate}
			\item $ \lambda_{H,G} \ne 0$,
			\item $p$ does not divide $|G^{\text{\'et}}{:}H^{\text{\'et}}|$ and $\lambda_{H,G^oH} \ne 0$,
		\item (only if $H \unlhd G^oH$) $p$ does not divide $|G^{\text{\'et}}{:}H^{\text{\'et}}|$ and $G^o/H^o$ is of multiplicative type.
		\end{enumerate}
		\end{proposition}
		
		\begin{proof}
			By the first part of Lemma~\ref{l:Q1} we  have $\lambda_{H,G} = \lambda_{H,HG^o} \lambda_{HG^o,G}$. By the second part $\lambda_{HG^o,G}=\lambda_{G^{\text{\'et}},H^{\text{\'et}}}$. This shows that $(1) \Leftrightarrow (2)$. 
			
			When $H \unlhd G^oH$ we have $\lambda_{H,G^oH} \sim \lambda_{1,G^o/H^o}$. But $\lambda_{1,G^o/H^o} \ne 0$ is equivalent to every $G^0/H^0$-module being projective relative to $1$, i.e.\ being projective, by Lemma~\ref{l:Q2}. It follows that every $G^o/H^o$-module is projective, which is equivalent to $G^o/H^o$ being of multiplicative type by \cite{MR1100646}.
		\end{proof}
		
		\begin{lemma}
			Suppose that $G$ is unipotent and $H \leq G$. Then $\lambda_{H,G} \ne 0$ if and only if $H=G$.
			\end{lemma}
		
	\begin{proof}
		We saw before that $\dim \Hom_G(k,\coind_H^G(k))=1$. Since $G$ is unipotent, this implies that $\dim \soc (\coind^G_H(k))=1$ and thus $\coind^G_H(k)$ is indecomposable. The only way that $k$ can be a summand is if $\coind^G_H(k) \cong k$.
		\end{proof}

\subsection{The double coset formula}\label{sec:Mackey}
In \cite{MR1375579}, the author states and proves a double coset formula for unimodular cocommutative Hopf algebras, generalizing Mackey's classical result for finite groups, see for example \cite[3.6.3(iv)]{benson}. In this subsection we present the translation of the formula from the language of Hopf algebras to that of group schemes, referring the reader to \cite[4.18 \& 4.19]{MR1375579} for the proof and more details.

Throughout this subsection, we assume that the ground field $k$ is algebraically closed, and that the group scheme $G$ as well as all of its subgroup schemes are unimodular. All transfer maps considered are defined on invariant subspaces of modules, as in comment (2) following Definition \ref{def:transfer}, but they are only well-defined up to non-zero scalar multiple.

Let $K,H$ be closed subgroup schemes of $G$. The quotient $G/H$ of the right action of $H$ on $G$, see \cite[I.5.5-I.5.6]{MR2015057}, is an affine $k$-scheme, \cite[III.\S3.5.4]{MR0302656}, with coordinate algebra $k[G/H]=\;^Hk[G]$. The left action of $K$ on $G$ induces a left action on $G/H$; one may then form the quotient $K \backslash G / H$, which is also an affine $k$-scheme, \cite[III.12, Theorem 1]{MumfordAbelian}, with coordinate algebra $k[K \backslash G / H]=\;^Hk[G]^K$.
\begin{definition}
	In the above notation, $K \backslash G / H$ will be referred to as a {\em double coset scheme}.
\end{definition}

\begin{example}\label{ex:Mackey}
	For a commutative $k$-algebra $A$, and $ a,b,c\in A$, let $M(a,b,c)$ denote the matrix
	\[
	\begin{pmatrix}
		1 &a & b\\
		0&1&c\\
		0&0&1
	\end{pmatrix}.
	\]
	Consider the finite group scheme $G$ that represents the group-valued functor
	\[
	A\mapsto G(A)=
	\left\{
	M(a,b,c)
	\suchthat a^p=b^p=c^p=0
	\right\},
	\]
	and the closed subgroup schemes $H,K$ of $G$ that represent the subfunctors
	\[
	A\mapsto H(A)
	=
	\left\{
	M(0,0,\gamma)
	\in G(A)
	\right\}
	\text{ and }
	A\mapsto 
	K(A)
	=
	\left\{
	M(0,\beta,\gamma)
	\in G(A)
	\right\}.
	\]
	The right action of $H$, respectively the left action of $K$, on $G$ are given by matrix multiplication, as follows.
	\[
	G(A)\times H(A)\mapsto G(A),\;\left(M(a,b,c),M(0,0,\gamma)\right)\mapsto M(a,b+a\gamma ,c+\gamma)
	\]
	\[
	K(A)\times G(A)\mapsto G(A),\;\left(M(0,\beta ,\gamma),M(a,b,c)\right)\mapsto M(a,b+\beta ,c+\gamma).
	\]
	One can then verify that $V:=k[G/H]=\;^Hk[G]=k[a,d]$, where $d=ac-b$, and that the action of $K$ on $G$ takes $d$ to $d+a\gamma-\beta$. For each $0\leq i\leq p-1$, the monomial $a^i\in k[a,d]$ is invariant under the induced action of $K$ and thus $k[a]\subseteq V^K$. For the reverse inclusion, note that if $K'$ is the closed subgroup scheme of $K$ that represents the subfunctor
	\[
	A\mapsto 
	K'(A)
	=
	\left\{
	M(0,\beta,0)
	\in G(A)
	\right\},
	\]
	then $k[d]\cong k[K']$ as a $K'$-module. Thus $V^{K}\subseteq  V^{K'}\cong (k[a]\otimes k[d])^{K'}\cong k[a]$ which gives $k[K \backslash G / H]=\;^Hk[G]^K= V^K=k[a]$.
	
\end{example}

Let $c \colon G\times G\rightarrow G$ be the conjugation morphism of $G$, see \cite[I.2.4(7)]{MR2015057}. Functorially, this is given by considering, for each commutative $k$-algebra $A$, the map
\[
{\rm Hom}_{k\text{-alg}}(k[G],A)\times {\rm Hom}_{k\text{-alg}}(k[G],A)
\rightarrow
{\rm Hom}_{k\text{-alg}}(k[G],A),\;
(g_1,g_2)\mapsto g_1g_2g_1^{-1}
\]
Recall that $G^0$ denotes the connected component of the identity of $G$, and $G^{\text{\'et}}$ the quotient $G/G^0$. If the ground field $k$ is perfect, then the canonical projection splits and $G^{\text{\'et}}$ may be identified with a closed subgroup scheme of $G$ (see \cite[6.8]{MR547117}). For a fixed point $x\in G^{\text{\'et}}(k)$, let $c_x \colon G\rightarrow G$ be the restriction of $c$ to $\{x\}\times G$.
\begin{definition} In the above notation, the schematic image of the restriction of $c_x$ to a closed subgroup scheme $H$ of $G$ will be referred to as the {\em conjugate} of $H$ by $x\in G^{\text{\'et}}$, and denoted by $H^x$.
\end{definition}
Recall from  \cite[I.2.6]{MR2015057} that given a closed subgroup scheme $H$ of $G$, the group-valued functor that maps a $k$-algebra $A$ to the normalizer of the finite group ${\rm Hom}_{k\text{-alg}}(k[H],A)$ in  ${\rm Hom}_{k\text{-alg}}(k[G],A)$ is represented by a finite group scheme $N_G(H)$, called the {\em normalizer} of $H$ in $G$. Functorially,
$$A \mapsto N_{G(A)}(H(A)) = \{ g \in G(A) \mid g H(A) g^{-1} = H(A) \}$$

\begin{definition} 
	Let $K,H$ be closed subgroup schemes of $G$. One says that $H$ is {\em normalized} by $K$ if $K$ is contained in $N_G(H)$. The minimal subgroup scheme of $G$ that contains $H$ and is normalized by $K$ will be denoted by $H^K$.
\end{definition}

\noindent {\em Remark.} The reader is warned not to confuse the notation $H^K$ above, which is consistent with that of classical group theory, with the invariants/fixed points $M^K$ of the action of a group scheme $K$ on a $K$-module $M$.

\begin{theorem}\label{th:Mackey}
	Let $K,H$ and $G$ be unimodular finite group schemes over an algebraically closed field $k$, with $K,H\leq G$, and let $M$ be a $G$-module. For each closed point $x$ of $K^{\text{\'et}} \backslash G^{\text{\'et}} / H^{\text{\'et}}$ there exists a map
	\[
	\gamma_x \colon
	{}^{H^{G^0}} \! \! M \rightarrow {}^{K\cap (H^{G^0})^x} \! M
	\]
	such that 
	\[
	{\rm res}_K^G{\rm Tr}_H^G
	=
	\sum_{x\in K^{\text{\'et}} \backslash G^{\text{\'et}} / H^{\text{\'et}}}
	{\rm Tr}_{K\cap (H^{G^0})^x}^K\circ \gamma_x\circ {\rm Tr}_H^{H^{G^0}}.
	\]
\end{theorem}

\begin{proof}
	See \cite[4.18 \& 4.19]{MR1375579}
\end{proof}

\noindent{\em Remark.} In light of the Remark preceding the theorem, $H^{G^0}$ is the minimal subgroup scheme of $G$ that contains $H$ and is normalized by $G^0$, while $M^{H^{G^0}}$ denotes the fixed points of $M$ under the action of $H^{G^0}$. Note that if $H^{G^0}=H$, the formula of Theorem \ref{th:Mackey} becomes identical to the classical Mackey formula for finite groups, see \cite[3.6.11]{benson}. \\

\noindent{\em Remark.} In \cite[4.18]{MR1375579}, one finds  ${\rm Tr}_{K\cap H^x}^K $ instead of ${\rm Tr}_{K\cap (H^{G^0})^x}^K$, but this is clearly not what is intended by the author: see the three lines marked {\em Proof of Theorem 4.19} on p.302. This change means that we might be transferring from larger subgroups than might be expected from the case of finite groups.\\

\noindent{\em Remark.} The transfer maps are only defined up to non-zero scalar multiple. One makes a choice and then the $\gamma_x$ depend on this choice.\\

\noindent{\em Remark.} Combining Theorem \ref{th:Mackey} with Higman's Criterion in Proposition \ref{prop:Higman} one sees that modules which are first induced from $H$ and then restricted back to $K$ are always projective relative to the class of subgroups of the form $K\cap (H^{G^0})^x$.\\

\noindent{\em Remark.} Another approach to the Mackey formula can be found in \cite[\S 8]{MR486168}.

\begin{example}
	We illustrate the necessity of the term $H^{G^0}$ in the formula above by exhibiting a case where projectivity fails relative to the naive intersection $K \cap H$, let $K',K,H,G$ be as in Example \ref{ex:Mackey}. 
	
	Assume for contradiction that $V=k[G/H]$ is projective relative to some proper subgroup scheme of $K$.  Recall that we have established that $V\cong k[a,d]$ and that $V/aV \cong k[d]$ is isomorphic to the regular representation of the infinitesimal subgroup scheme $K'\cong \alpha_p$; this forces $V/aV$ to be indecomposable as a $K$-module. Since $a$ is a nilpotent invariant element, it is contained in the radical and so $V$ itself must be an indecomposable $K$-module. Thus $V$ must be a summand of a module induced from a proper subgroup of $K$. 
	
	Since $K\cong \alpha_p\times\alpha_p$, its proper subgroup schemes are either trivial or isomorphic to $\alpha_p$. Indecomposable $\alpha_p$-modules are of dimension $1,\ldots,p$  and their induction to $K$ gives modules of dimension $p,\ldots,p^2$. Thus, the only induced module large enough to contain $V$ as a summand is free, and the same is true for the trivial subgroup. In any case, $V$ must be projective, contradicting the fact that $\dim_k V^K= p$. It follows that $V$ has no summand of the form $k[K/H]$, as would be expected from the case of finite groups. Note that this is consistent with the result of Theorem \ref{th:Mackey}, since $H$ is not normal in this case.
\end{example}

\subsection{Cohomology}\label{sec:coh}
Let $M$ be a $G$-module. For $n\in\mathbb{N}$, we denote as usual by ${\rm Ext }_G^n(M,-)$ the $n$-th right derived functor of the left exact functor ${\rm Hom}_G(M,-)$, \cite[2.5.2]{Weibel}. Group scheme cohomology can be defined either as $H^n(G,-)={\rm Ext }_G^n(k,-)$, or as the $n$-th right derived functor of the left exact functor ${}^G(-)={\rm Hom}_G(k,-)$, \cite[I.4.2]{MR2015057}.

For properties referring to ${\rm Ext }_G^n$, resp. $H^n(G,-)$, for all $n\in\mathbb{N}$ we shall often simply write ${\rm Ext }_G$, resp. $H(G,-)$. Note that ${\rm Ext}_G$ is a {\em universal} $\delta$-functor, see \cite[2.6]{Weibel}; thus any morphism defined on the level of ${\rm Hom}$ sets gives rise to a unique morphism between {\rm Ext} groups.

\begin{definition}\label{def:exttransfer}
	Let $H$ be a closed subgroup scheme of $G$, and let $M,M'$ be $G$-modules. The unique map of $k$-modules
	\[
	{\rm Ext}_H(\omega_{H,G}^{-1}\otimes M',M)\rightarrow{\rm Ext}_{G}(M',M)
	\]
	that extends the transfer map ${\rm Hom}_H(\omega_{H,G}^{-1}\otimes M',M)\rightarrow{\rm Hom}_{G}(M',M)$ of Definition \ref{def:transfer} will also be referred to as the {\em relative transfer} from $H$ to $G$ and denoted by ${\rm Tr}_{H}^G$.
\end{definition}

\noindent{\em Remark.} Equivalently, one can arrive at Definition \ref{def:exttransfer} by taking a projective resolution of $M'$,
\[
(P_\bullet,d_\bullet) \colon \cdots \overset{d_2}{\longrightarrow} P_1\overset{d_1}{\longrightarrow} P_0\overset{d_0}{\longrightarrow} M'.
\]
For each $n\in\mathbb{N}$, the transfer maps $
{\rm Hom}_H(\omega_{H,G}^{-1}\otimes P_{n}, M)\rightarrow  {\rm Hom}_G(P_{n}, M)
$ of Definition \ref{def:transfer} commute with the differentials $d_n$ of the resolution by Proposition \ref{pr:properties-tr}. In this manner, one obtains a map of cochain complexes ${\rm Hom}_H(\omega_{H,G}^{-1}\otimes P_{\bullet}, M)\rightarrow {\rm Hom}_G(P_{\bullet}, M)$ and the relative transfers of Definition \ref{def:exttransfer} are exactly the induced maps on cohomology.\\

Note that when both $G$ and $H$ are unimodular, see Definition \ref{def:delta},  choosing an isomorphism $M' \cong k$ in Definition \ref{def:exttransfer} gives rise to a map $H(H,M)\rightarrow H(G,M)$ in group scheme cohomology. In particular, if $G$ is a constant group scheme, one retrieves the usual definition of transfer for group cohomology; this follows by Remark (3) after Definition \ref{def:transfer} and universality of the ${\rm Ext}$ functor. However, when $\omega_{H,G}\not \cong  k$, the transfer is not defined between cohomology groups with trivial coefficients; instead, one has the following two special cases.

\begin{proposition} Let $H$ be a closed subgroup scheme of $G$.
	\begin{enumerate}
		\item
		Definition \ref{def:exttransfer} gives rise to transfer maps
		\[
		H (H,\omega_{H,G})\rightarrow H (G,k).
		\]
		\item If $\delta_H$ is the restriction of some $G$-module $\chi $, then Definition \ref{def:exttransfer} gives rise to transfer maps
		\[
		H (H,k)\rightarrow H (G,\delta_G \chi^{-1}).
		\]
	\end{enumerate}
\end{proposition}
\begin{proof}
	Putting $M=M'=k$ in Definition \ref{def:exttransfer}, the codomain of the transfer becomes ${\rm Ext}_G\left(k,k\right)=H(G,k)$; by \cite[I.4.4]{MR2015057}, its domain satisfies
	$
	{\rm Ext}_H^n(\omega_{H,G}^{-1},k)\cong {\rm Ext}_H^n(k,\omega_{H,G})=H (H,\omega_{H,G})$, proving (1). For (2), put $M'=k$ and $M=\delta_G\chi^{-1}$ in Definition \ref{def:exttransfer} to get a tranfer with domain ${\rm Ext}_H^n(k,k)=H^n(H,k)$ and codomain ${\rm Ext}_H^n(k,\delta_G\chi^{-1})=H^n(G,\delta_G\chi^{-1})$.
\end{proof}

\noindent{\em Remark.} The $G$-module $\chi $ above always exists when $G$ acts on $H$, as is the case when $H$ is normal in $G$ and $G$ acts via conjugation, see \cite[I.8.19]{MR2015057}.
\\

The properties below follow directly from Proposition \ref{pr:properties-tr} and universality of ${\rm Ext}$.

\begin{proposition}\label{prop:exttrprop}
	Let $M',M,N',N$ be $G$-modules, and let $K,H$ be closed subgroup schemes of $G$ with $K\leq H$.
	\begin{enumerate}
		\item If $f\in{\rm Ext}_H^{n_1}(M'\otimes\omega_{H,G}^{-1},M),\;g\in{\rm Ext}_G^{n_2}\left(N', M'\right)$ and $h\in{\rm Ext}_G^{n_3}\left(M,N\right)$, then 
		\[
		h\circ {\rm Tr}_{H}^G(f)\circ g={\rm Tr}_{H}^G(h \circ f\circ (g\otimes\mathbb{1}_{\omega_{H,G}^{-1}}))
		\in
		{\rm Ext}_G^{n}\left(N',N\right),
		\]
		where $n=n_1+n_2+n_3$, $\circ$ is the Yoneda product, and $\mathbb{1}_{\omega_{H,G}^{-1}}$ is the identity element of ${\rm Ext}_H(\omega_{H,G}^{-1},\omega_{H,G}^{-1})$.
		\item ${\rm Tr}_{H}^G\circ{\rm Tr}_{K}^H={\rm Tr}_{K}^G$.\\
		\item If $L/k$  is a field extension, and $G_L=G\times_{{\rm Spec} k}{\rm Spec}L,\;H_L=H\times_{{\rm Spec} k}{\rm Spec}L$, then ${\rm Tr}_{H}^G\otimes L={\rm Tr}_{H_L}^{G_L}$, as maps
		$
		{\rm Ext}_{H_L}((M'\otimes L)\otimes_L\omega_{H_L,G_L}^{-1},M\otimes L))\rightarrow
		{\rm Ext}_{G_L}(M'\otimes  L,M\otimes  L).
		$
	\end{enumerate}
\end{proposition}

The properties of the transfer that we proved for modules carry over to ${\rm Ext}$ groups. In particular, $\Tr^G_H \res^G_H$ is multiplication by $\lambda_{H,G}$.

\section{The norm}\label{sec:norm}
\subsection{Notation}\label{sec:notation}
Throughout this section, $k$ is a field of characteristic $p>0$, $G$ is a finite group scheme over $k$ and $S$ is a commutative $k$-algebra endowed with the structure of a right $k[G]$-comodule algebra via a coaction map $\sigma:S\rightarrow S\otimes k[G]$, see \cite[4.1.2]{MR1243637}. 

\subsection{Invariants of infinitesimal group schemes}\label{sec:norminf}

Recall from \cite[II.7.1.4]{MR0302656} that if $G$ is infinitesimal and $\mathfrak{m}$ denotes the maximal ideal of the local ring $k[G]$, then the {\em height} of $G$ is
\[
{\rm ht}(G)
={\rm min}
\{
n\;|\;f^{p^n}=0\;\text{for all}\;f\in\mathfrak{m}
\}.
\]

\begin{proposition}\label{prop:height}
	Let $G$ be an infinitesimal group scheme. Then
	\begin{enumerate}
		\item $s^{p^n}\in S^{G}$ for all $s\in S$ and all $n\geq{\rm ht}(G)$.
		\item $s^{|G|}\in S^{G}$ for all $s\in S$.
		\item $s^{|G{:}H|}\in S^G$ for all closed subgroup schemes $H\leq G$ and all $s\in S^H$.
	\end{enumerate}
\end{proposition}
\begin{proof}
	(2) follows from (1) since the order of $G$ is a $p$-th power, \cite[11.30]{MR3729270} and $\mathfrak{m}^{|G|}=0$, \cite[11.32]{MR3729270}. (1) follows from \cite[1.5.2]{MR4844642} or \cite[2.8]{MR4692505}; we recall the proof from the latter for clarity of the exposition. Let $\epsilon:k[G]\rightarrow k$ be the augmentation map of the Hopf algebra $k[G]$. Each $f\in k[G]$ can be written as $f=\epsilon(f)+\overline{f}$, with $\overline{f}\in\ker\epsilon$. Then $(\ker\epsilon)^{p^n}=0$, for all $n\geq{\rm ht}(G)$, and so $f^{p^n}=\epsilon(f)^{p^n}$ for all $f\in k[G]$.
	If $\sigma(s)=\sum_{i,j} s_i\otimes f_j$ is the coaction map of $k[G]$ on $S$, then for all $s\in S$
	\[
	\sigma(s^{p^n})
	=\sigma(s)^{p^n}
	=\sum_{i,j} s_i^{p^n}\otimes f_j^{p^n}
	=\sum_{i,j} s_i^{p^n}\otimes \epsilon(f_j)^{p^n}
	=\sum_{i,j} s_i^{p^n}\epsilon(f_j)^{p^n}\otimes  1
	=s^{p^n}\otimes 1,
	\]
	where the last equality is the second comodule axiom, \cite[I.2.8(3)]{MR2015057}. Hence $s^{p^n}\in S^{G}$.
	
	For (3), let $n={\rm ht}(G)$. By \cite[I.9.4(1)]{MR2015057}, there is a filtration of $G$ by normal subgroup schemes
	\[
	G=G_0>G_1>\cdots>G_n=\{e_G\},
	\]
	with $G_i$ infinitesimal for all $0\leq i\leq n$, \cite[I.9.6(1)]{MR2015057}. We proceed by induction on $n$. 
	
	If $n=1$, then $s^p\in S^G$ for all $s\in S$ by part (1). Assume that the statement is true for all $i\leq n-1$.
	\begin{itemize}
		\item If $G_{n-1}\leq H$, then $s\in S^{G_{n-1}}$ for all $s\in S^H$, and $S^{G_{n-1}}$ is a $G/G_{n-1}$-module, \cite[I.6.4]{MR2015057}. As ${\rm ht}(G/G_{n-1})<{\rm ht}(G)=n$, one may apply the inductive assumption with $G$ replaced by  $G/G_{n-1}$, $H$ replaced by $H/G_{n-1}$ and $S$ replaced by $S^{G_{n-1}}$ to get that
		\[
		s^{|G/G_{n-1}{:}H/G_{n-1}|}\in (S^{G_{n-1}})^{G/G_{n-1}}.
		\]
		The result follows since $(S^{G_{n-1}})^{G/G_{n-1}}=S^G$ by \cite[I.6.4]{MR2015057} and  $|G/G_{n-1}{:}H/G_{n-1}|=|G{:}H|$ by the third isomorphism theorem, \cite[5.55]{MR3729270}.
		\item If $G_{n-1}\nleq H$ then for all $s\in S^H$ one has that $s^{|G_{n-1}H{:}H|}\in S^{G_{n-1}}\cap S^{H}=S^{G_{n-1}H}$ by part (1). Applying the inductive assumption with $G$ replaced by $G/G_{n-1}$, $H$ replaced by $G_{n-1}H/G_{n-1}$, $S$ replaced by $S^{G_{n-1}}$ and $s$ replaced by $s^{|G_{n-1}H{:}H|}$ yields
		\[
		\left(s^{|G_{n-1}H{:}H|}\right)^{|G/G_{n-1}{:}G_{n-1}H/G_{n-1}|}\in (S^{G_{n-1}H})^{(G/G_{n-1})}.
		\]
		As above, $(S^{G_{n-1}H})^{(G/G_{n-1})}=S^G$ by \cite[I.6.4]{MR2015057}, and $|G/G_{n-1}{:}G_{n-1}H/G_{n-1}|=|G{:}G_{n-1}H|$ by \cite[5.55]{MR3729270}. Thus $$|G/G_{n-1}{:}G_{n-1}H/G_{n-1}||G_{n-1}H{:}H|=|G:G_{n-1}H||G_{n-1}H{:}H|=|G{:}H|$$ and the result follows.
	\end{itemize}
\end{proof}

\noindent {\em Remark.}
The result above can be used to show that rings of invariants of finite group schemes are finitely generated, see \cite[2.9]{MR4692505}; note that this proof is considerably shorter than other classical proofs, as presented in \cite[\S4.2]{MR1243637}.\\

In the result below we do not assume $G$ to be infinitesimal, and we denote as usual the connected component of its identity by $G^0$.
\begin{corollary}\label{cor:height}
	For any $H\leq G$ and all $s\in S^H$, $s^{|G^0H{:}H|}\in S^{G^0H}$ .
\end{corollary}
\begin{proof}
	Let $s\in S^H$, so that $s\in S^{H^0}$. By Proposition \ref{prop:height}, $s^{|G^0{:}H^0|}\in S^{G^0}$ and so $s^{|G^0{:}H^0|}\in S^{G^0}\cap S^H$. By \cite[I.6.2(3)]{MR2015057}, $|G^0H{:}H|=|G^0{:}(G^0\cap H)|=|G^0{:}H^0|$ and thus $s^{|G^0H{:}H|}\in S^{G^0}\cap S^H=S^{G^0H}$.
\end{proof}

\subsection{Definition of the norm}\label{sec:defnorm}
Let $G^0$ be the connected component of the identity of $G$, and let $G^{\text{\'et}}$ be the quotient $G/G^0$.  If $k^{\text{sep}}$ is a separable closure of $k$, then the base change $G^{\text{\'et}}\times_{{\rm Spec}(k)}{\rm Spec}(k^{\text{sep}})$ may be identified with the finite group $\Gamma_G=G(k^{\text{sep}})$ of $k^{\text{sep}}$-rational points of $G$, \cite[6.4]{MR547117}, which acts on $S\otimes k^{\text{sep}}$. For any closed subgroup scheme $H$ of $G$, $\Gamma_H=H(k^{\text{sep}})$ is a subgroup of $\Gamma_G$ and thus one can consider the group-theoretic Evens norm, \cite[6.1]{Evens},
\begin{equation}\label{eq:evensnorm}
	{\rm norm}_{\Gamma_H,\Gamma_G} \colon \left(S\otimes k^{\text{sep}}\right)^{\Gamma_H}\rightarrow\left( S\otimes k^{\text{sep}}\right)^{\Gamma_G},\; s\otimes \kappa\mapsto \prod_{\gamma \in T_{G/H}}\gamma (s\otimes \kappa),
\end{equation}
where $T_{G/H}$ is a right transversal of $\Gamma_H$ in $\Gamma_G$.

\begin{proposition}\label{prop:norminv}
	In the above notation, one has for all $s\in S^H$ that
	\[
	{\rm norm}_{\Gamma_H,\Gamma_G}\left((s\otimes 1_{k^{\text{sep}}})^{|G^0H{:}H|}\right)
	=
	\prod_{\gamma \in T_{G/H}}\gamma (s\otimes 1_{k^{\text{sep}}})^{|G^0H:H|}
	\in S^G
	\]
\end{proposition}
\begin{proof}
	Let $s\in S^H$. By Corollary \ref{cor:height}, $s^{|G^0H{:}H|}\in S^{G^0H}$, and so $(s\otimes 1_{k^{\text{sep}}})^{|G^0H{:}H|}\in S^{G^0H}\otimes k^{\text{sep}}$.  $S^{G^0}\otimes k^{\text{sep}}$ is a $\Gamma_G$-module and thus eq.\   (\ref{eq:evensnorm}) gives
	\[
	{\rm norm}_{\Gamma_H,\Gamma_G}\left((s\otimes 1_{k^{\text{sep}}})^{|G^0H{:}H|}\right)
	=
	\prod_{\gamma \in T_{G/H}}\gamma (s\otimes 1_{k^{\text{sep}}})^{|G^0H{:}H|}
	\in ( S^{G^0H}\otimes k^{\text{sep}})^{\Gamma_G}.
	\]
	By \cite[I.2.10(3)]{MR2015057} $( S^{G^0}\otimes k^{\text{sep}})^{\Gamma_G}=( S^{G^0})^{G^{\text{\'et}}}\otimes k^{\text{sep}}=S^G\otimes k^{\text{sep}}$, and thus it remains to show invariance under the action of the absolute Galois group $\mathcal{G}={\rm Gal}(k^{\text{sep}}/k)$. For $\phi\in\mathcal{G}$,
	\begin{equation}\label{eq:normgal}
		\phi\prod_{\gamma \in T_{G/H}}\gamma (s\otimes 1_{k^{\text{sep}}})^{|G^0H{:}H|}=
		\prod_{\gamma \in T_{G/H}}\phi(\gamma)\phi( (s\otimes 1_{k^{\text{sep}}})^{|G^0H{:}H|})=
		\prod_{\gamma \in T_{G/H}}\gamma (s\otimes 1_{k^{\text{sep}}})^{|G^0H{:}H|},
	\end{equation}
	where the last equality follows since $\phi$ permutes the set of cosets $\Gamma_G/\Gamma_H$ and $s\otimes 1_{k^{\text{sep}}}$ is $\mathcal{G}$-invariant.
\end{proof}

\begin{definition}\label{def:norm}
	The {\em (relative) norm} from $H$ to $G$ is the map
	\[
	{\rm Nm}_H^G \colon S^H\rightarrow S^G,\; 
	s\mapsto{\rm norm}_{\Gamma_H,\Gamma_G}\left((s\otimes 1_{k^{\text{sep}}})^{|G^0H:H|}\right).
	\]
\end{definition}

\subsection{Properties of the norm.}\label{sec:propnorm}
The properties below can be seen as generalizations of well-known properties of the Evens norm from the context of finite groups, to that of finite group schemes.

\begin{proposition}\label{pr:norm-properties}
	\begin{enumerate}
		\item[]
		\item ${\rm Nm}_H^G(s)=s^{|G{:}H|}$ for all $s\in S^G$.
		\item If $K,H$ are closed subgroup schemes of $G$ with $K\leq H$, then ${\rm Nm}_H^G\circ {\rm Nm}_K^H={\rm Nm}_K^G$.
		\item ${\rm Nm}_H^G(rs)={\rm Nm}_H^G(r){\rm Nm}_H^G(s)$, for all $r,s\in S^H$.
		\item If  $f \colon S\rightarrow R$ is a $G$-algebra homomorphism, then ${\rm Nm}_H^G\left(f(s)\right)=f\left({\rm Nm}_H^G(s)\right)$  for all $s\in S^H$.
		\item If $\ell$  is an extension field of $k$ and we consider the $\ell$-algebra $S \otimes_k \ell$ we have ${\rm Nm}^G_H (s \otimes 1) = {\rm Nm}^G_H(s)$.
	\end{enumerate}
\end{proposition}

\begin{proof}
	(1): If $s\in S^G$ then ${\rm Nm}_H^G(s)=(s^{|\Gamma_G {:}\Gamma_H|})^{|G^0H{:}H|}$, see \cite[Note on p. 58]{Evens}. Note that $|G^0H{:}H|=|G^0{:}H^0|$; the result follows since
	\[
	|\Gamma_G{:}\Gamma_H||G^0H{:}H|=|G^{\text{\'et}}{:}H^{\text{\'et}}||G^0{:}H^0|=|G{:}H|.
	\]
	
	(2): If $s\in S^K$ then
	\begin{eqnarray*}
		({\rm Nm}_H^G\circ {\rm Nm}_K^H)(s)&=&
		{\rm norm}_{\Gamma_H,\Gamma_G}
		\left[
		\left({\rm Nm}_K^H(s)\otimes 1_{k^{\text{sep}}}\right)^{|G^0H{:}H|}
		\right]
		\\
		&=&
		{\rm norm}_{\Gamma_H,\Gamma_G}
		\left[
		\left({\rm norm}_{\Gamma_K,\Gamma_H}\left((s\otimes 1_{k^{\text{sep}}})^{|H^0K{:}K|}\right)\otimes 1_{k^{\text{sep}}}\right)^{|G^0H{:}H|}
		\right]
		\\
		&=&
		({\rm norm}_{\Gamma_H,\Gamma_G}
		\circ
		{\rm norm}_{\Gamma_K,\Gamma_H})
		\left((s\otimes 1_{k^{\text{sep}}})^{|H^0K{:}K||G^0H{:}H|}\right)
		\\
		&=&
		{\rm norm}_{\Gamma_K,\Gamma_G}\left((s\otimes 1_{k^{\text{sep}}})^{|G^0K{:}K|}\right)
		=
		{\rm Nm}_K^G(s),
	\end{eqnarray*}
	where the last equality follows by transitivity of the Evens norm \cite[6.1.1(N1)]{Evens} and the fact that 
	\[
	|G^0H{:}H||H^0K{:}K|=|G^0{:}H^0||H^0{:}K^0|=|G^0{:}K^0|=|G^0K{:}K|.
	\]
	
	(3) follows by multiplicativity of the Evens norm \cite[6.1.1(N2)]{Evens} and (4) because the Evens norm commutes with equivariant homomorphisms of algebras. (5) is trivial.
\end{proof}

\subsection{Mumford's norm}\label{sec:Mumford}
In \cite[pp.112-113]{MumfordAbelian}, Mumford defines a variant of the norm map. Let $G$ be a finite group scheme over $k$ and let $S$ be a commutative $k$-algebra endowed with the structure of a right $k[G]$-comodule algebra via a coaction map $\sigma \colon S\rightarrow S\otimes k[G]$. Let $m \colon S\otimes k[G]\rightarrow {\rm End}_S\left(S\otimes k[G]\right)$ map $f\in S\otimes k[G]$ to the endomorphism given by multiplication by $f$ and let ${\rm det}$ denote the usual determinant.
\begin{definition}\label{def:Mumfordnorm}
	{\em Mumford's norm} ${\mathcal{N}}^G \colon S\rightarrow S^G$ is the composition
	\[
	S\overset{\sigma}{\rightarrow}S\otimes k[G]\overset{m}{\rightarrow}{\rm End}_S\left(S\otimes k[G]\right)\overset{{\rm det}}{\rightarrow} S.
	\]
\end{definition}

\noindent{\em Remark.} The fact that the image of ${\mathcal{N}}^G$ is invariant under $G$ is non-trivial; a proof can be found in loc.\ cit.. Also, even though Mumford assumes $k$ to be algebraically closed, his definition makes sense not only over arbitrary fields, but also over commutative rings, as shown in \cite[(4.18)-(4.21), pp. 56-57]{AbelianBook}. For a treatment in the language of Hopf algebras, see  \cite{MR1275919} and \cite[4.2.2]{MR1243637}.\\

\begin{proposition}\label{prop:Mumford}
	Let $S$ be a commutative $k$-algebra endowed with the structure of a right $k[G]$-comodule algebra. Then
	${\mathcal{N}}^G(s)={\rm Nm}^G(s)$ for all $s\in S$.
\end{proposition}

\begin{proof}
	Determinants commute with scalar extension, thus $\mathcal{N}^G$ does so as well. By Proposition \ref{pr:norm-properties}, the same is true for ${\rm Nm}^G$, so we may assume $k$ is algebraically closed. This allows us to identify as before the \'etale quotient $G^{\text{\'et}}$ with the finite group of rational points $\Gamma_G=G(k)$ and to obtain, by \cite[6.8]{MR547117}, a semi-direct product decomposition $G \cong G^0 \rtimes \Gamma_G$, with $\Gamma_G$  viewed as a subgroup of $G$. For each $\gamma \in \Gamma_G$, let $G_\gamma$ be the connected component of $G$ containing $\gamma$, and let $A_\gamma = k[G_\gamma]$; note that translation by $\gamma^{-1}$ induces an isomorphism $A_\gamma \cong k[G^0]$. The coordinate algebra $k[G]$ then decomposes as a product of these local algebras
	$
	k[G]\cong \prod_{\gamma\in\Gamma_G} A_\gamma,
	$
	which induces an $S$-module decomposition
	$
	S\otimes k[G]=\prod_{\gamma\in\Gamma_G}S\otimes A_\gamma,
	$
	with each $S\otimes A_\gamma$ a free $S$-module of rank $|G^0|$. If $s \in S$ and $m_s$ denotes the $S$-endomorphism of $S\otimes k[G]$ given by multiplication by $\sigma(s)$ and $m_{s,\gamma}$ is its restriction to $S\otimes A_\gamma$, then  
	$
	m_s=\prod_{\gamma\in\Gamma_G} m_{s,\gamma}
	$
	and so
	\[
	{\mathcal{N}}^G(s)={\rm det}(m_s)=\prod_{\gamma\in\Gamma_G} {\rm det}(m_{s,\gamma}).
	\]
	It thus suffices to show that ${\rm det}(m_{s,\gamma})=\gamma(s)^{|G^0|}$. Let $\mathfrak{m}_\gamma$ be the maximal ideal of $A_\gamma$. Since the latter is a finite local algebra over an algebraically closed field, its residue field is $A_\gamma/\mathfrak{m}_\gamma \cong k$, \cite[7.9]{Atiyah}. By  \cite[I.2.8]{MR2015057}, the action of the $k$-rational point  $\gamma\in\Gamma_G=G(k)$ on $s\in S$ is given by
	\[
	\gamma(s) = ({\rm id}_S \otimes {\rm ev}_\gamma)(\sigma(s))
	\]
	where ${\rm ev}_\gamma\colon k[G]\twoheadrightarrow k$ is the evaluation map at $\gamma$. This may be identified with the composition
	\[
	S\xrightarrow{\sigma} S\otimes k[G] \xrightarrow{{\rm id}_S\otimes\pi_\gamma} S\otimes A_\gamma \xrightarrow{{\rm id}_S\otimes {\rm ev}_\gamma} S\otimes k \cong S,
	\]
	and so the endomorphism $m_{s,\gamma}$ of $S\otimes A_\gamma$ acts on the quotient  $(S\otimes A_\gamma)/(S\otimes \mathfrak{m}_\gamma)\cong S$ as multiplication by $\gamma(s)$. Thus $m_{s,\gamma}$ acts on $S\otimes A_\gamma$ by multiplication by $\gamma(s)\cdot 1 + \eta$, with  $\eta\in S\otimes \mathfrak{m}_\gamma$ and the term $\gamma(s)\cdot 1$ acting as scalar multiplication. Since $\mathfrak{m}_\gamma$ is nilpotent, we can choose a basis for $S\otimes A_\gamma$ compatible with the filtration $S\otimes A_\gamma \supset \mathfrak{m}_\gamma (S\otimes A_\gamma) \supset \dots \supset 0$. With respect to such a basis, the action of $\eta$ is strictly triangular with zeros on the diagonal, and thus the matrix of $m_{s,\gamma}$ is triangular with diagonal entries all equal to $\gamma(s)$, giving the desired
	\[
	{\rm det}(m_{s,\gamma})=\gamma(s)^{{\rm rank}_S(S\otimes A_\gamma)}=\gamma(s)^{|G^0|}.
	\]
\end{proof}

\subsection{The field norm} \label{sec:fieldnorm}
In the notation of Section \ref{sec:notation}, assume that $S=L$ is a field acted upon by $G$. Applying Definition \ref{def:norm} with $H$ trivial, gives rise to a norm map 
\[
{\rm Nm}^G \colon L\rightarrow L^G,\;s\mapsto  \prod_{\gamma \in \Gamma_G}\gamma (s\otimes 1_{k^{\text{sep}}})^{|G^0|}.
\] 
One can also consider the classical {\em field norm}, see \cite[page 456]{MR3677125},
\begin{equation}\label{eq:normother2}
	{\rm N}_{L/L^G} \colon L\rightarrow L^G,\; s\mapsto {\rm det}(m_s),
\end{equation}
where $m_s \colon L\rightarrow L,\;t\;{\mapsto}st$ is the $L^G$-endomorphism of $L$ given by multiplication by $s$.

\begin{proposition}
	In the above notation, $[L{:}L^G]$ divides $|G|$ and ${\rm Nm}^G(s)={\rm N}_{L/L^G}(s)^{\frac{|G|}{[L{:}L^G]}}$ for all $s\in L$.
\end{proposition}

\begin{proof}
	Let $H$ be the kernel of the action of $G$ on $L$; then $G/H$ acts faithfully on $L$, and so by \cite[8.3.7]{MR1243637} 
	\[
	\dim_k k[G]^H=\dim_k k[G/H]= [L{:}L^{G/H}]=[L{:}L^G].
	\]
	By the Nichols-Zoeller theorem, see \cite[3.1.5]{MR1243637}, $\dim_k k[G]^H=[L{:}L^G]$ divides $\dim_kk[G]=|G|$ and $k[G]$ is free over $k[G]^H$ of rank $\frac{|G|}{[L{:}L^G]}$.
	
	By \cite[III.\S 9.1.(12)]{zbMATH05069335} one may extend scalars from $L^G$ to $L$, view the field norm as the determinant of multiplication by $1 \otimes s$ on the $L$-module $L \otimes_{L^G} L$ and rewrite eq. (\ref{eq:normother2}) as
	\[
	{\rm N}_{L/L^G}(s)=
	\det\left(L\xrightarrow{\cdot s} L\right)
	=
	\det \left( L \otimes_{L^G} L \xrightarrow{\cdot (1 \otimes s)} L \otimes_{L^G} L \right).
	\]
	Consider the $L$-map
	$
	{\beta}: L \otimes_{L^G} L \longrightarrow L \otimes_k k[G], \; x \otimes y \longmapsto x\sigma(y)
	$, see  \cite[8.1.1]{MR1243637}; since $H$ acts trivially on $L$, $\sigma(s)\in L\otimes_k k[G]^H$ and so $\beta$ factors as 
	\[
	L \otimes_{L^G} L \xrightarrow{\bar{\beta}} L \otimes_k k[G]^H\hookrightarrow L\otimes_k k[G].
	\]
	As $G/H$ acts faithfully on $L$, \cite[8.3.7]{MR1243637} implies that $\bar{\beta}$ is an isomorphism of $L$-algebras, mapping $1 \otimes s$ to $\sigma(s)$. Thus
	\[
	{\rm N}_{L/L^G}(s)
	=
	\det \left( L \otimes_{L^G} L \xrightarrow{\cdot (1 \otimes s)} L \otimes_{L^G} L \right)
	=
	\det \left( L \otimes_k k[G]^H \xrightarrow{\cdot \sigma(s)} L \otimes_k k[G]^H \right).
	\]
	Finally, recalling from the first paragraph that $L \otimes_k k[G]$ is free over $L \otimes_k k[G]^H$ of rank $\frac{|G|}{[L{:}L^G]}$, one has by  \cite[III.\S 8.6.(33)]{zbMATH05069335}  that
	\begin{align*}
	{\rm N}_{L/L^G}(s)^{\frac{|G|}{[L{:}L^G]}}
	&=
	\det \left( L \otimes_k k[G]^H \xrightarrow{\cdot \sigma(s)} L \otimes_k k[G]^H \right)^{\frac{|G|}{[L:L^G]}}\\
	&=
	\det \left( L \otimes_k k[G] \xrightarrow{\cdot \sigma(s)} L \otimes_k k[G] \right)=\mathcal{N}^G(s),
	\end{align*}
	where $\mathcal{N}^G(s)$ is Mumford's norm of Definition \ref{def:Mumfordnorm}. The result follows since ${\rm Nm}^G(s)=\mathcal{N}^G(s)$ by Proposition \ref{prop:Mumford}.
\end{proof}

\noindent{\em Remark.} 
One may alternatively prove directly that ${\rm Nm}^G$ coincides with the following equivalent definition of the field norm of a finite extension $L/K$,   see for example \cite[VI.\S5]{MR1878556},
\[
{\rm N}_{L/K} \colon L\rightarrow K,\;
s\mapsto
\prod_{\sigma\in{\rm Emb}_K(L,\overline{K})}\sigma(s)^{[L{:}K]_i},
\]
where ${\rm Emb}_K(L,\overline{K})$ are the distinct $K$-embeddings of $L$ in an algebraic closure of $K$ and $[L{:}K]_i$ is the inseparable degree of $L/K$.

\section{The Wirthm\"uller Isomorphism} \label{wirt}

Although we did not use the Wirthm\"uller isomorphism (Theorem~\ref{th:wirth}) to define the transfer map, we did use it to develop the properties of the transfer and it is an important feature of the representation theory of finite group schemes. A statement of the result can be found in \cite[I.8.17]{MR2015057}; it does not mention that the isomorphism is natural, but this can be deduced from the proof. All the ingredients of the proof also appear in \cite{MR347883}.

This is an example of a phenomenon often observed in symmetric monoidal categories. A very general account is given in  \cite[Example 4.6]{MR3542492}, although it only considers the stable category. 

Here we sketch a proof that makes some of the properties of the isomorphism transparent.

\begin{lemma}
	For finite group schemes $H \leq G$, we have two natural transformations of functors from $H$-modules to $G$-modules, given on an $H$-module $M$ as follows:
	\begin{enumerate}
		\item 
		$$\coind^G_H (M) = k {}_GG_H  \otimes_{kH}  M \cong \Hom_{H_1} (  \Hom _{H_2} (k {}_G G_{H_2}, k {}_{H_1}H_{H_2}), {}_{H_1}M),$$
		by $g \otimes m \mapsto (f \mapsto f(g)m)$ for $g \in kG$, $m \in M$, $f \in \Hom _{H_2} (k {}_G G_{H_2}, k {}_{H_1}H_{H_2})$.
		\item 
		$$\ind_H^G(M) = \Hom_H(k {}_H G_G , {}_H M) \cong \Hom_{H_1} (  \Hom _{H_2} (k [{}_G G_{H_2}], k [{}_{H_1}H_{H_2}]), {}_{H_1}M).$$
	This is induced by isomorphisms 
	$$k{}_{H_1}G_G \cong k[ {}_G G_{H_1} ]^* \cong \Hom _{H_2} (k {}_G G_{H_2}, k {}_{H_1}H_{H_2});$$
	the first is just the double dual together with the antipode to switch sides, the second is $\eva_1 \circ f \mapsfrom f$, where $\eva_1$ denotes evaluation at 1.
\end{enumerate}
\end{lemma}

\begin{proof}
	We leave it to the reader to check the required equivariance properties of these maps.
	
	Part (1) is in \cite[I 1.16]{MR2015057}. We claim that the same formula gives an isomorphism
	$$ N \otimes_H M \rightarrow \Hom_{H_1} (  \Hom _{H_2} (N, k {}_{H_1}H_{H_2}), {}_{H_1}M),$$
	for any finitely generated projective right $H_2$-module $N$. Any finitely-generated projective is a summand of a finitely-generated free module and both sides commute with finite direct sums, so we are reduced to the case $N \cong kH$. which is clear.
	
	The second isomorphism in part (2) is because $k[H] \cong (kH)^* \cong \ind_1^H(k)$ as right $H$-modules and this is an instance of the adjoint property of induction.
\end{proof}

The next lemma is from \cite[I 8.12 Remark 1]{MR2015057}.

\begin{lemma}
	$kG$ has the structure of a $k[G]$-module and as such  the multiplication map $k[G] \otimes {}_G \delta \rightarrow kG$ gives an isomorphism of $G{-}G$-bimodules (with the diagonal action on the left hand side).
\end{lemma}

Using this lemma we find
$$ \Hom _{H_2} (k {}_G G_{H_2}, k {}_{H_1}H_{H_2}) \cong \Hom _{H_2} (k [{}_G G_{H_2}] {}_G \delta , k [{}_{H_1}H_{H_2}] {}_{H_1} \delta ) \cong \Hom _{H_2} (k [{}_G G_{H_2}], k [{}_{H_1}H_{H_2}]) {}_G \delta^{-1} {}_{H_1} \delta .$$
Combining this with the previous lemma, we obtain the Wirthm\"uller isomorphism.

This formulation makes it straightforward to check identities such as the following transitivity result for $K \leq H \leq G$.

$$ W^G_K(M) = \coind ^G_H ( \ide_{\omega^{-1} _{H,G}} \otimes W^H_K(M)) \circ W^G_H( \ind ^H_K (M)) = W^G_H(\coind^H_K (\omega^{-1} _{K,H}M)) \circ \ind^G_H(W^H_K(M)).$$

\bibliographystyle{alpha}
\bibliography{KKGeneral.bib}

\end{document}